\documentclass[a4paper, 12pt]{amsart}
\usepackage{amssymb,amsmath,txfonts,mathrsfs,titletoc,appendix}
\usepackage{graphicx}
\usepackage{latexsym}
\usepackage[alphabetic]{amsrefs}
\usepackage{geometry}
\usepackage{fancyhdr}
\usepackage{xcolor}
\usepackage{comment}
\usepackage{subcaption}
\usepackage{pifont} 

\usepackage{exscale}        
\usepackage{relsize}       

\usepackage{tensor}   

\usepackage{float} 

\usepackage[new]{old-arrows}
\usepackage{longtable}
\usepackage{booktabs}
\usepackage{array}
\usepackage{multirow}
\usepackage{supertabular}
\usepackage{amssymb}
\usepackage{amsmath}
\usepackage{mathrsfs}
\usepackage{titletoc}
\usepackage{latexsym}
\usepackage[alphabetic]{amsrefs}

 \newtheorem{thm}{Theorem}[section]
 \newtheorem{cor}[thm]{Corollary}
   
  \newtheorem{cjt*}{Conjecture}
 \newtheorem{lem}[thm]{Lemma}
 \newtheorem{prop}[thm]{Proposition}
 \newtheorem{defn}[thm]{Definition}
 \newtheorem{rem}[thm]{Remark}
 \numberwithin{equation}{section}
\newtheorem{lem*}{Lemma}
\newtheorem{cor*}{Corollary}

\newenvironment{pf}{\medskip \noindent
{\bf Proof of Theorem \ref{mainthm1}.}}{\hfill \rule{.5em}{1em}
\\}

\newenvironment{rmk}{\mbox{ }\\{\bf  Remark}\mbox{ }}{
\hfill $\Box$\mbox{}\bigskip}

\def\R{\Bbb{R}}
\def\C{\Bbb{C}}

\begin{document}

\title{Delaunay hypersurfaces in spheres}
\author{Yongsheng Zhang}
\address{Academy for Multidisciplinary Studies, Capital Normal University, Beijing 100048, P. R. China}
\email{yongsheng.chang@gmail.com}
\date{\today}

\keywords{Delaunay hypersurface, CMC, spiral product, flower type, embeddedness} 
\begin{abstract}
We study Delaunay hypersurfaces in $\mathbb S^n$ with $n\geq 3$ and add a missing (flower) type of the category.
Moreover, embedded Delaunay hypersurfaces of nonzero constant mean curvatures in $\mathbb S^n$ are found.
\end{abstract}
\maketitle
\section{Introduction}\label{S1} 
In 1841, Delaunay discoverd a wonderful way of constructing rotational hypersurfaces of constant mean curvature (CMC) in $\mathbb R^3$.
All rotational CMC hypersurfaces come from the rolling construction of roulettes of conics.

The Delaunay's rolling construction was successfully generalized to the case of CMC rotationally hypersurfaces in $\R^n$ with $n> 3$ in \cite{H-Y},
and later to the cases in hyperbolic space $\mathbb H^n$ and standard Euclidean sphere $\mathbb S^n$ respectively with $n\geq 3$ in \cite{H}.
                %
                            
                        %
                           %
                            %
                            %
                     \begin{figure}[h]
	\centering
	\begin{subfigure}[t]{0.3\textwidth}
		\centering
		\includegraphics[scale=0.54]{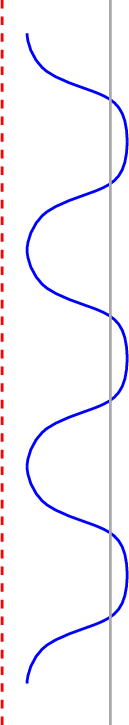}
                              \captionsetup{font={scriptsize}} 
                             \caption{Unduloid}
                               \label{fig:1a}
	\end{subfigure}
	\begin{subfigure}[t]{0.3\textwidth}
		\centering
	 \includegraphics[scale=0.54]{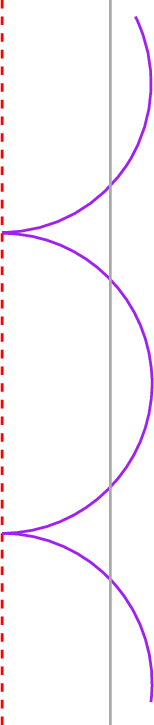}\,\,\,\,\,\,
                     \captionsetup{font={scriptsize}} 
                       \caption{Union of spheres}\label{fig:1b}
	\end{subfigure}
	\begin{subfigure}[t]{0.3\textwidth}
		\ \ \ \,
	 \includegraphics[scale=0.54]{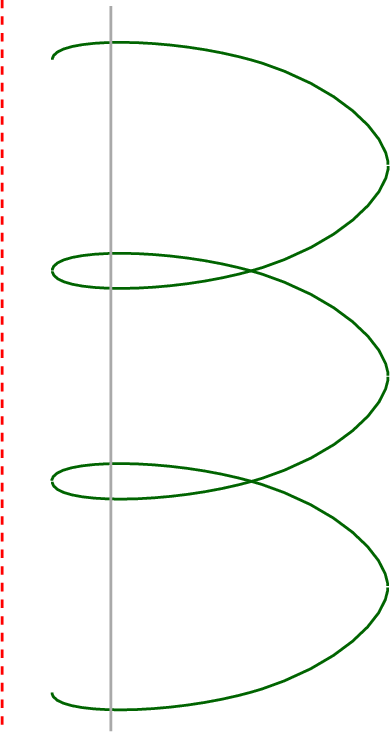}
                     \captionsetup{font={scriptsize}} 
                             \caption{Nodoid}\label{fig:1c}
	\end{subfigure}
	\caption{Delaunay $h$-CMC hypersurfaces with $h>0$} \label{fig:1}
\end{figure}
                               Let the dashed line stand for a geodesic and the gray line present the unique tubular hypersurface of the center geodesic of constant mean curvature $h>0$ 
                               with respected to the inner unit normal vector field of the solid tube.
                               Then, by increasing the largest distance from the center geodesic,
                               the transformation of Delaunay $h$-CMC hypersurfaces can be illustrated from left to right in Figure \ref{fig:1}.

                            A recent joint paper \cite{L-Z}
                            systematically studies
                            spiral minimal products.
                            A somehow degenerate but useful situation in \cite{L-Z} 
                            is the so-called singly spiral product
                            and we shall use this kind of product to rebuild spherical Delaunay hypersurfaces directly in this paper. 
                            More explicitly, 
                            taking $M_1=\{point\ (1,0)\}\in \mathbb S^1\subset \C^1$ and $M_2=\mathbb S^{n-2}$
                            with a singly spiral curve
                            \begin{equation}\label{n0}
                            \gamma(t)=\big(a(t)e^{it},\, b(t)\big)\in \mathbb S^2,
                            \  \text{ where  domain } I \text{ of } t \text{ is an open interval of}\ \R^1,
                            \end{equation}
                            we consider
\begin{eqnarray}
\nonumber
G_\gamma:\, I\times S^{n-2} &\longrightarrow& \mathbb S^n\subset \C^1\oplus \R^{n-1}\\
(t,\, x)\ \ \,\,\, &\longmapsto&\ \, \big(a(t)e^{it},\, b(t)x\big)\, .
\label{n1}
\end{eqnarray}
                            An advantage of this (local) construction of Delaunay $h$-CMC hypersurfaces
                            is to visualize things on $\mathbb S^2$.
                            
                            With the singly spiral product \eqref{n1}, we complete the category of Delaunay hypersurfaces in $\mathbb S^n$ by adding a missing type (see Theorem \ref{flower} and Corollary \ref{complete}).
                           
                            The structure of paper is organized as follows.
                            In \S \ref{S2} 
                            we recall basics about singly spiral products  from \cite{L-Z} and solve the $h$-CMC equation locally.
                            Local pieces of $h$-CMC hypersurfaces are described in \S \ref{S3},
                            while assembling to global $h$-CMC hypersurfaces 
                            as well as a missing type (named flower type) 
                            of negative constant mean curvature
                            are given in \S \ref{S4}.
                            In \S \ref{S5} 
                            we show that there exist infinitely many choices of negative  $h$
                            such that each the flower type  Delaunay hypersurface
                             induces an $h$-CMC immersed closed submanifold in the target sphere.
                             Moreover, we get the existence of uncountably many choices of positive $h$ each of which allows an embedded $h$-CMC Delaunay submanifold in the sphere.

                       {\ }
                       
                       \section{The $h$-CMC equation}\label{S2}
                       The $h$-CMC equation is said to the curve $\gamma$ in \eqref{n0} for $G_\gamma$ to be of constant mean curvature $h\in \R$,
                       and a (local) solution curve $\gamma$ is then called an $h$-curve.
                       For simplicity, $h$ in this paper always means the value of unnormalized mean curvature, i.e., the trace of second fundamental form.
                       
                      Note that for any $h\in \R$, there exists a combination $(a, b)\in \R^2_+$ 
                       such that $a\cdot\mathbb S^1\times b\cdot\mathbb S^{n-2}$ has constant mean curvature $h$ in $\mathbb S^n$.
                       Such combination is uniquely determined by
                     \begin{equation}\label{n00}
                       \frac{a}{b}=\dfrac{\ h+\sqrt{h^2+4(n-2)}\ }{2(n-2)}.
                       \end{equation}
                       See Theorem 3 (ii) of \cite{H} or \S 2 of \cite{L-Z}.
                       
                       From now on we focus on $C^2$ immersed $\gamma\subset \mathbb S^2$ with varying $(a(t),\, b(t))$ over some open interval  $I$.
                       According to \S 3 of \cite{L-Z},
                       the unit normal vector field we described in \S \ref{S1}
                       for the construction \eqref{n1}
                       can be gained by the normalization of 
                       \begin{equation}\label{te0} 
                        \tilde \eta_0(t,x)=
                          \Big(b(t)e^{it},\, -a(t)x\Big) 
                         -
                          \frac{\mathcal V}
                        {         %
                         \Theta
                          }
                         \Big( ia(t)e^{it},\, 0\Big) 
\end{equation}
where $\mathcal V=a'b-ab'$, $\Theta=(as'_1)^2$ and $\|G'_\gamma\|^2=\Theta+(a')^2+(b')^2$.
                 Let $\big\{v_1,\, \cdots, v_{n-2}\big\}$ be a local orthonormal basis around $x$ of $\mathbb S^{n-2}$.
                 Then $ \big\{ 
         \left(0,  v_1\right),\,\cdots,\, \left(0,  v_{n-2}\right), \, \tilde E
         \big\}$ 
         where $\tilde E
                    = 
                                            \frac{G'_\gamma}
                                            {
                          \left
                          \|
                          G'_\gamma
                          \right
                          \|
                            }
$ provides a local orthonormal basis around $G_\gamma(t, x)$.
By consulting Lemma 3.5 of \cite{L-Z} we know
 that     the second fundamental form of $G_\gamma$ in $\mathbb S^n$ is given by 
 $$A_{\tilde \eta_0}
                         =
                                \begin{pmatrix}
                                \frac{a}{b}I_{n-2}
                                                              & *\\
                                 * & \tilde \boxast
                                \end{pmatrix}
                                $$
                                where
                                           $$
     \left\|G'_\gamma\right\|^2 \cdot \tilde \boxast
                    =
                                   %
                                          \big[
                                          a''-a\cdot (t')^2
                                          \big]
                                          b
                                                                                 -
                                                 a
                                                       b''
                                                                                                  -
                                       \frac{\mathcal V}{\Theta}
                                           \big\{
                                              2a a' \cdot (t')^2
                                          \big \}.
$$
%
     Therefore,          
     an $h$-curve for generating an $h$-CMC hypersurface is characterized by      
                      $$
                        \|G'_\gamma\|^2
                                \cdot 
                        \|\tilde \eta_0\|
                                  \cdot
                         h
                         = \|G'_\gamma\|^2
                                 \cdot 
                                         \Big(\tilde \boxast+(n-2)\frac{a}{b}\Big).
                      $$
           
                       Of course one may try to solve the equation directly.
                       However, since our geometric concern is independent of choice of parametrization,
                       as explained in \cite{L-Z}, if we use arc parameter $s$ for curve $(a(t), b(t))$ in $\mathbb S^1$ with $a(t)=\cos s$ and $b(t)=\sin s$ from the very beginning,
                       namely $$\gamma(s)=\big(a(s)e^{is_1(s)},\, b(s)\big),$$
                       then  
                                  $\mathcal V=-1$,
                                        $\Theta=(a\dot s_1)^2$,
                       $ \|\dot G_\gamma\|^2=1+(a\dot s_1)^2 $,
                       and with 
                       $$  \tilde \eta_0(t,x)=
                          \Big(b(s)e^{is_1(s)},\, -a(s)\dot s_1(s)x\Big) 
                         +
                          \frac{1}
                        {         %
                         \Theta
                          }
                         \Big( ia(s)\dot s_1(s)e^{is_1(s)},\, 0\Big) 
                       $$
                       we have
                       $ \|\tilde \eta_0\|=\sqrt{\frac{1+\Theta}{\Theta}}$.
                       Now the $h$-CMC requirement  becomes
                       $$
                  \sqrt{\frac{1+\Theta}{\Theta}} h=(n-2)\frac{a}{b}
                       -\frac{ab(\dot s_1)^2}{1+\Theta}+\frac{-2ab(\dot s_1)^2+a^2\dot s_1\ddot s_1}{\Theta(1+\Theta)}
                       $$
                       which simplifies to the following $h$-CMC equation
                       \begin{equation}\label{e01}
        %
         -\frac{b}{a}+(n-2)\frac{a}{b}
                  -
       \sqrt{\frac{1+\Theta}{\Theta}}h
       =-
                \frac{\dot \Theta}{2\Theta  \left(
         1+\Theta
         \right)}\, .
             %
\end{equation}
As a result, we have
$$
\exp\left(2\int\frac{b}{a}-(n-2)\frac{a}{b}+\sqrt{\frac{1+\Theta}{\Theta}}h\, ds\right)=\frac{C_1 \Theta}{1+\Theta}\, 
$$
for some $C_1\in\R_+$ and hence, with $\Gamma=\frac{\Theta}{1+\Theta}$, it follows that
               $$
                         \exp\left(2\int{\frac{h}{\sqrt \Gamma}}\, ds\right)=\tilde C_1 \left(\cos^2 s \sin ^{2(n-2)}s \cdot\Gamma\right)
                              $$
                             for some $\tilde C_1\in \R_+$ and therefore
                                $$
                               \frac{2h}{\sqrt \Gamma} \exp\left(2\int{\frac{h}{\sqrt \Gamma}}\, ds\right)=\tilde C_1 \frac{d}{ds}\left(\cos^2 s \sin ^{2(n-2)}s \cdot\Gamma\right).
                                $$
                                So, we get
                                $$
                                 \frac{2h}{\sqrt \Gamma}=\frac{\frac{d}{ds}\left(\cos^2 s \sin ^{2(n-2)}s \cdot\Gamma\right)}{\cos^2 s \sin ^{2(n-2)}s \cdot\Gamma}
                                $$
                                and consequently,
                                $$
                               \left( \cos s \sin ^{n-2} s \right) h=\frac{d}{ds}\left(\sqrt{\cos^2 s \sin ^{2(n-2)}s \cdot\Gamma}\, \right).
                                $$
                                Finally, the $h$-CMC equation \eqref{e01} is solved 
                                by
                                $$
                                C+\frac{h}{n-1}\sin^{n-1}s= \sqrt{\cos^2 s \sin ^{2(n-2)}s \cdot\Gamma},
                                \text{\ \ \  where } C\in \R.
                                $$
                                Thus, 
      \begin{equation}\label{n3.5}
                                \Theta=
                                {\dfrac{\left(C+\frac{h}{n-1}\sin^{n-1}s\right)^2}{\cos^2s \sin^{2(n-2)}s -\left(C+\frac{h}{n-1}\sin^{n-1}s\right)^2}}
  \end{equation}
                                and
                 \begin{equation}\label{n4}
                                \dot s_1=
                                \pm
                                          \sqrt{
                                                   \dfrac{
                                                              \left(C+\frac{h}{n-1}\sin^{n-1}s\right)^2}
                                                              {
                                                              \cos^2s
                                                              \left[\cos^2s \sin^{2(n-2)}s -\left(C+\frac{h}{n-1}\sin^{n-1}s\right)^2                                                   
                                                              \right]
                                                              }
                                                }\, \ .
              \end{equation}
                           {\ }
                                                        
                       {\ }
                       
                       \section{Local construction of $h$-CMC Delaunay hypersurfaces}\label{S3}
                       
                       In this section, we  will build local pieces of $h$-curves in variable $s$.
                       As we shall see, 
                      sometimes $\Theta$ could touch zero inside $(0,\frac{\pi}{2})$ and,
                      due to the $\pm$ choices in \eqref{n4},  
                      there can exist the situation where by geometric meanings the expression $\dot s_1$ in  \eqref{n4} needs to take alternating signs on two sides of the (interior) zero of $\Theta$.
                      By default we shall draw pictures for the choice of non-negative $\dot s_1$.

                      Making sense  of \eqref{n3.5}  with non-negativity of the denominator  simply requires  that
                                $$
                                {\cos s \sin^{n-2}s >C+\frac{h}{n-1}\sin^{n-1}s}\, .
                                $$
                                Denote the left term of the inequality by $L(s)$ and the right by $R(s, C, h)$.
                                Clearly, 
                                for $n\geq 3$,
                                $L$ reaches its maximal at $s_0=\arctan \sqrt{n-2}\in(0,\frac{\pi}{2})$ and 
                      the derivatives of both sides with respect to $s$ are     
                                $$
                                \dot L=\big[-\tan s +(n-2)\cot s\big]\cos s\sin^{n-2}s
                                \text{\ \ \ and \ \ } 
                                \dot R=h\cos s\sin^{n-2}s.
                                $$
                                          \begin{figure}[ht]  
                                	\includegraphics[scale=1.1]{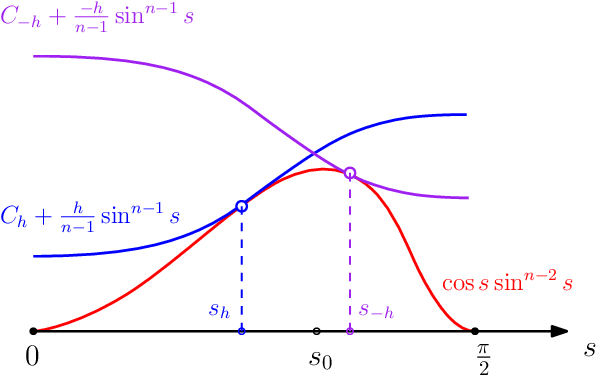}
                             \caption{Graphs of $R(s, C, \pm h)$ for $h>0$ and contact points}
                               \label{fig:2}
 \end{figure}
So,  when $h>0$ ($h=0$ or $h<0$),
             there exists a unique point $s_h\in (0, s_0)$ ($s_h=s_0$ or $s_h\in (s_0, \frac{\pi}{2})$) 
                                such that $\dot L(s_h)=\dot R(s_h, h)$, i.e.,
                                $$
                              -\tan s_h +(n-2)\cot s_h=h.
                                $$
                                This indicates that,
                                 no matter $h<0$, $h=0$ or $h>0$,
                                there exists a unique $C_h$ such that graphs of $L$ and $R$ with $C=C_h$ contact at only one point 
                                and exactly $\frac{a}{b}=\cot s_h$ fulfills \eqref{n00}.

               {\ }
               
                       Consequently, 
                       the largest possible domain $I\subset (0,\frac{\pi}{2})$ for an $h$-curve with  \eqref{n4} valid (i.e. $\Theta>0$)
                        is a connected interval depending of the choice of $C$ as in Figure \ref{fig:3}.
         \begin{figure}[ht]  
                                	\includegraphics[scale=1.1]{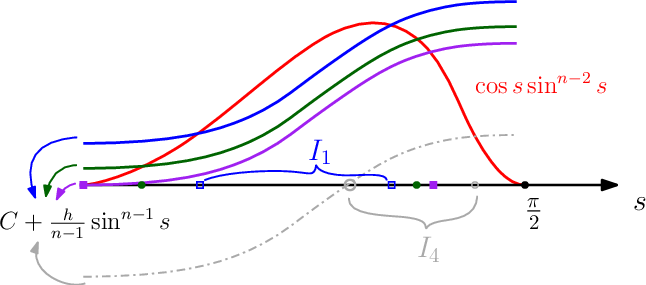}
                             \caption{Lagest domain for $\Theta>0$ in \eqref{n3.5} when $h>0$}
                               \label{fig:3}
 \end{figure}
  Since the behavior types will be seen differently, we state the following two propositions separately.
 
       \begin{prop}
       Assume that $h>0$.
       If  $C\in [0, C_h)$, then the largest possible open domain $I(C, h)$ is decided by two intersection points of closures of graphs of $R(\cdot\,, C, h)$ and $L(\cdot)$.
       When  $C\in (-\frac{h}{n-1},0)$, distinctly $I(C, h)$ is determined by  intersection points of the graph of $R(\cdot\, , C, h)$ with the $s$-axis and the graph of $L(\cdot)$ respectively.
       \end{prop}

\begin{prop}\label{ppp2}
     Assume that $h>0$.
       If  $C\in [\frac{h}{n-1}, C_{-h})$, then the largest possible open domain $I(C, -h)$ is decided by two intersection points of 
       closures of graphs of $R(\cdot\,, C, -h)$ and $L(\cdot)$.
       When  $C\in (0,\frac{h}{n-1})$, distinctly $I(C,-h)$ is determined by  intersection points of the graph of $R(\cdot\,, C, -h)$ 
       with the graph of $L(\cdot)$ and the $s$-axis  respectively.
       \end{prop}
                      {\ }

                       \section{Global construction of $h$-CMC Delaunay hypersurfaces}\label{S4}
                       In this section we  shall exhibit how to assemble local pieces of $h$-curves for a global $h$-CMC Delaunay hypersurface
                       based on the understanding from \S \ref{S3}.
                       There are essentially three types of behaviors of $h$-curves for every non-vanishing value $h\in \R$.
                       For positive $h$, they are similarly as shown in Figure \ref{fig:1}.
                       
                       Whereas, for negative $h$, one missing limit type is found as the counterpart of union of spheres instead.
                       We name it  flower type since its generating $h$-curve crosses the peculiar point $p=(0,0,1)$ infinitely many times
                       and an aerial view (see Figure \ref{fig:12}) over $p$ consists of copies of petals.
                       It will be seen that,
                       up to a phase gauge in $s_1$ (always ignored in our consideration), 
                       the generating $h$-curve of flower type is unique. 
                       Tracing the curve,
                        each petal corresponds to a connected piece over $s_1$,
                       but ``adjacent" petals has jumping domains in $s_1$
                     and this forms a reason why the limit type has not been discovered for decades.
                     
                       The main idea of global constructions is to apply natural reflection/rotational constructions according to Figures \ref{fig:2} and \ref{fig:3}.
                        Let us explain details in below for $h$ of different signs separately.

                       \subsection{Case of positive $h$. }
                       As shown in Figure \ref{fig:3},
                       when $C\in (0, C_h)$,
                       there is a largest open domain
                       $$\left(z_L^h,\, z_R^h\right):=I(C, h)\Subset \left(0,\, \frac{\pi}{2}\right)$$
                       for $\Theta$ in \eqref{n3.5} to be positive.
                       
                       By $D$ we mean $\cos^2s \sin^{2(n-2)}s -(C+\frac{h}{n-1}\sin^{n-1}s)^2$,
                        the denominator of $\Theta$ in  \eqref{n3.5}.
                       Then, first of all, we would like to mention a simple observation.
                            \begin{prop}\label{pp4}
                            For $C\in (0, C_h)$,
                            the derivative of $D$ with respect to $s$ is nonzero at $z_L^h$ and $z_R^h$.
                            \end{prop}
                            \begin{proof}
                            This is clear according to Figure \ref{fig:2}.
                            \end{proof}
                       
                       This result ensures that the integral of $\dot s_1$ in \eqref{n4} over $(z_L^h,\, z_R^h)$ is finite
                       and hence one can reflect the curve at the ending points along $s_1$-direction to get the following figure.
                         \begin{figure}[ht]  
                                	\includegraphics[scale=0.8]{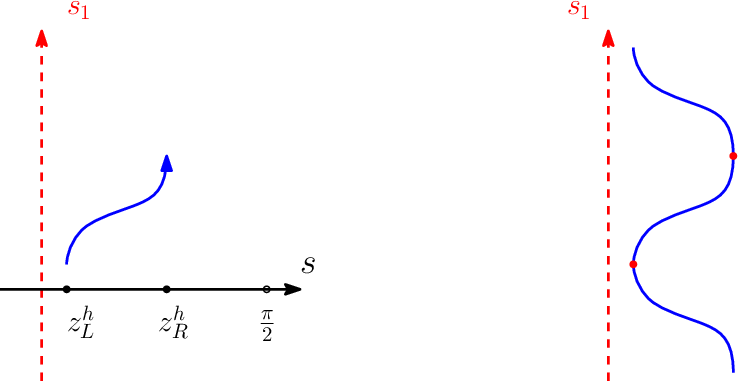}
                             \caption{Reflection extension when $h>0$ and $C\in (0, C_h)$ }
                               \label{fig:4}
 \end{figure}

                              Now consider the differentiability of the extension at joint points.
                              Obviously, it is $C^1$ at joint points,
                              so one can get a $C^0$ inner unit vector field.
                              In fact, the extended curve has much stronger differentiability for its geometric meanings.
                              
                               \begin{prop}\label{pp5}
                               For $h>0$ and $C\in (0, C_h)$,
                               the repeatedly extended curve as in Figure \ref{fig:4} 
                               gives an analytic $h$-curve 
                               and consequently a corresponding immersion of constant mean curvature $h$
                               from $\R\times  S^{n-2}$ into $\mathbb S^n$ is gained.
                               \end{prop}
                               \begin{proof}
                              As in both sides of a joint point
                              the local $h$-curve generates hypersurface of the same constant mean curvature $h$
                              and a joint point is an extremal point in Figure \ref{fig:4},
                              it follows that the extended curve is indeed $C^2$ at each joint point.
                              Furthermore, based on the $C^2$-differentiability,
                              Morrey's regularity theory \cite{M1, M2}
         asserts that the extended $h$-curve is analytic everywhere.
         \end{proof}
         

                       
                              For $C=0$,
                              it can be seen  from \eqref{n4} that
                              $\dot s_1$ vanishes at $s=0$
                              and we get the following.
       \begin{figure}[ht]  
                                	\includegraphics[scale=0.8]{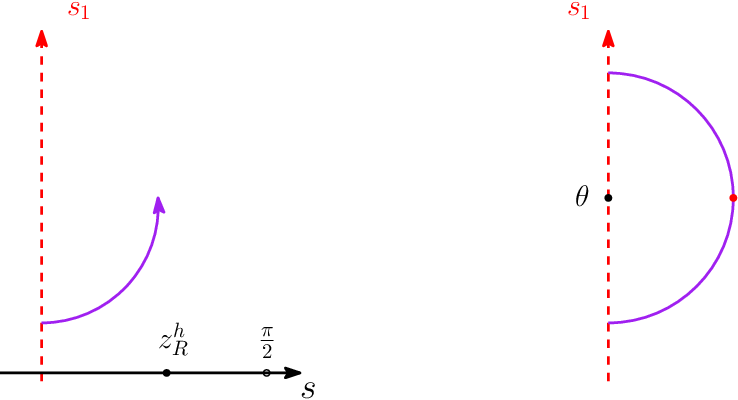}
                             \caption{Reflection extension for $C=0$ with $h>0$}
                               \label{fig:09}
 \end{figure}

                  \begin{prop}\label{pp6}
                  For $h\geq0$ and $C=0$,
                   via applying the reflection construction in the $s_1$-phase one time,
                               one gets an analytic $h$-curve as in Figure \ref{fig:09} 
                               which generates a hypersphere 
                               of size $\sin z^h_R$ with constant mean curvature $h$.
                               \end{prop}
                     \begin{proof}
                     Evidently, the local $h$-curve and its reflection copy together induce an embedded hypersphere.
                     This is impossible for non-vanishing $C$, cf. Figure \ref{fig:3}.
                    As $h>0$,
                    there is certain (round) sphere in $\mathbb S^n$ of mean curvature $h$
                     which can be generated by an $h$-curve.
                     So, by Figure \ref{fig:3}, 
                     the $h$-curve must be the curve in Figure \ref{fig:09}.
                     This curve induces an $h$-CMC $(n-1)$-sphere 
                    centered at $\nu$ of radius $\sin z^h_R$ in 
                            the affine space $\R^n_\nu=
                             \nu^\perp 
                                                                        \subset \R^{n+1}_\nu
                                                                        $
                                                                        where
                                                                        $
                                                                        \nu$
                                                                        stands for the point
                                                                          $\left(
                                           \cos z^h_R
                                                       \cdot 
                                                               (e^{i\theta}, 0, \cdots, 0)
                                                               \right)\in \R^{n+1}
                                                               $.
                                                               
                                                               As the mean curvature the hypersphere in $\mathbb S^n$ is $(n-1)\frac{\cos z^h_R}{\sin z^h_R}$,
                                one has 
                                $z^h_R=\arctan\frac{n-1}{h}$.
                                The conclusion then also holds for $h=0$.
              %
                     \end{proof}

Next, let us focus on the situation of $C\in (-\frac{h}{n-1},0)$.
Deriving from  the curve corresponding to $I_4$ in Figure \ref{fig:3},
we flip its part below the $s$-axis in the $ss_1$-plane to the above to get the domain $J_4=\left(z_L^h,\, z_R^h\right)$ as in Figure \ref{fig:10}.
                                                \begin{figure}[ht]  
                                	\includegraphics[scale=1.1]{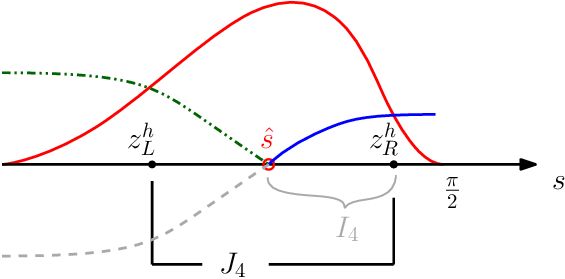}
                             \caption{interval $J$ for $C\in (-\frac{h}{n-1},0)$ with $h>0$}
                               \label{fig:10}
 \end{figure}
 The corresponding extension has one more step because of the occurrence of $\hat s$.
      \begin{figure}[ht]  
                                	\includegraphics[scale=0.7]{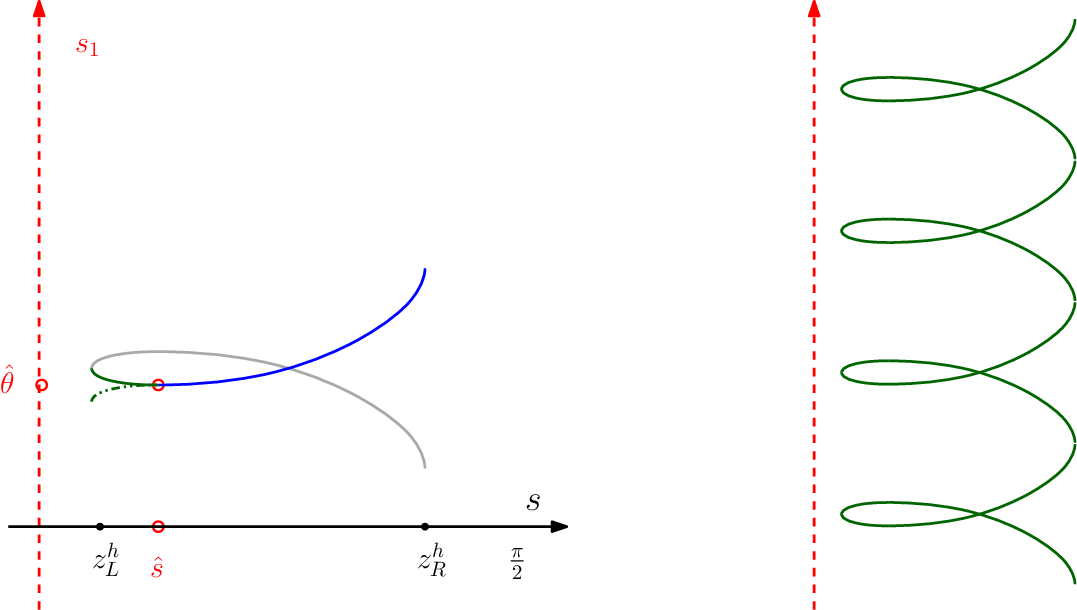}
                             \caption{Reflection extension when $h>0$ and $C\in (-\frac{h}{n-1},0)$ }
                               \label{fig:11}
 \end{figure}
 It can be seen that the dash dot dotted curve segment in Figure \ref{fig:11} stands for a hypersurface of constant mean curvature $-h$ 
 under the chosen unit normal vector field 
 and consequently its flipped mirror (solid curve in green defined over the $s$-range $(z^h_L,\, \hat s$) in Figure \ref{fig:11}) 
 with respect to $\{\hat \theta\}\times (0, \frac{\pi}{2}) $ corresponds to a hypersurface of constant mean curvature $h$. 
 Similarly as argued in Proposition \ref{pp5} we have the following.
                           \begin{prop}\label{pp7}
                           For $h>0$ and $C\in (-\frac{h}{n-1},0)$,
                   via mirror flips and  reflections along $s_1$-direction,
                               one gets an analytic $h$-curve as  in Figure \ref{fig:11} 
                               which generates a hypersurface
of constant mean curvature $h$, with whirls occurring in local model.
                               \end{prop}
\begin{rem}
Up to a sign, $\dot s_1$ is given by $ 
                                                   \dfrac{
                                                              C+\frac{h}{n-1}\sin^{n-1}s}
                                                              {
                                                              \cos s
                                                              \sqrt{
                                                              \cos^2s \sin^{2(n-2)}s -\left(C+\frac{h}{n-1}\sin^{n-1}s\right)^2                                                   
                                                              }
                                                }
$
over the entire $J$.
\end{rem}

                        \subsection{Case of vanishing $h$. }
                        This case falls exactly into the singly spiral minimal product of a point and $\mathbb S^{n-2}$ in $\mathbb S^{n}$
                        and
                       it has been systematically studied in  \cite{L-Z}.
                       Note that the rotational extension (of reflections in both slots) to construct global doubly spiral minimal products 
                       now reduces to the aforementioned (phase-reflection) extension in a single slot of \eqref{n1}.
                       In particular,  for a singly spiral minimal product the parameter $\tilde C$ in (3.25) of \cite{L-Z} is precisely $1/{C^2}$ for our $C$ here.
                        
                        Since all moving curves in Figures \ref{fig:2} and \ref{fig:3} become horizontal lines $\{s_1=C\}$ (thus not transversally hitting the $s$-axis),
                        whirls as in Figure \ref{fig:11} can never happen 
                        when $h=0$
                         and only oscillations along $s_1$-direction can  take place.
                        An exception is  $C=0$ which leads to a totally geodesic hypersphere.
                        With the end value $s_1$ of local $h$-curve fixed, as $C\downarrow 0$ one period of oscillations, 
                        namely the union of the local $h$-curve and its $s_1$-phase reflection,
                        will limit to a pair of geodesics joint at $p=(0,0,1)$ in $\mathbb S^2$, 
                        e.g.   $(\pm \cos s ,\, 0,\, \sin s)$ for $s\in[0, \frac{\pi}{2}]$.
                        The angle $\pi$ between the pair was in fact computed in Lemma 10.1 of \cite{L-Z} with $k_1=0$ therein.
                        
                         \subsection{Case of negative $h$. }
                         Unlike the Euclidean situations, the case of Delaunay hypersurfaces  in spheres can have oscillating phenomena for negative $h$ as well. 
                         Moreover, an amusing missing type of Delaunay hypersurfaces in spheres of negative constant mean curvature $h$ will be explored.
                         The philosophy to predict this 
                         is the following.
                          For $h>0$, like in Figure \ref{fig:1} from unduloid  to noduloid the deformation passes through union of spheres.
                          As we shall see that, for $h<0$, there are also ``unduloid" and  ``noduloid" types.
                         So, similarly, there must be one additional transition type connecting  them.
                         
                         To be consistent with Figures \ref{fig:1} and \ref{fig:2} and Proposition \ref{ppp2} 
                         we assume $h>0$ and denote the constant mean curvature by $-h$.
                         Using previously given symbols we arrive at the following.
                         
                           \begin{prop}\label{pp8}
                               When $C\in (\frac{h}{n-1}, C_{-h})$,
                               similarly as in Figure \ref{fig:4}
                               one can get an analytic 
                               $(-h)$-curve oscillating along $s_1$-dircetion 
                               which induces an immersion of constant mean curvature $-h$
                               from $\R\times S^{n-2}$ into $\mathbb S^n$;
                                when  $C\in (0,\frac{h}{n-1})$,
                                similarly as in Figure \ref{fig:11}
 one gets an analytic $(-h)$-curve 
                               which generates a hypersurface
of constant mean curvature $h$, with whirls occurring in local model.
                               \end{prop}
                               \begin{proof}
                               The same as the arguments for  Propositions \ref{pp5} and \ref{pp7}.
                               \end{proof}
                               \begin{rem}
                               The oscillating curves form unduloid type
                               and  those with whirls in local mode give noduloid type.
                               It should be mentioned that
                               the noduloid type is the exactly same  as in Figure \ref{fig:11}
                               but with the other unit normal vector field automatically induced by the construction (i.e., decided by the solid curve part in Figure \ref{fig:10}).
                               However, by checking oscillating ranges in $s$,
                               it is clear that unduloids for negative and non-negative $h$ are completely different sets.
                               \end{rem}

                       {\ }
                       
                       The missing type occurs when $C=\frac{h}{n-1}$ which can be viewed as a limit type.
                       By fixing plus sign in \eqref{n4} 
                       we have the following asymptotic result.
                       
                       \begin{lem}\label{l1}
                       For $h>0$ and  $C=\frac{h}{n-1}$,
                       every $(-h)$-curve in local model (see Figure \ref{fig:2}) limits to $p=(0,0,1)$
                       and along the curve
                       it follows that
                       $\lim_{s\uparrow \frac{\pi}{2}}\dot s_1=\frac{h}{2}$.
                       \end{lem}
                       \begin{proof}
                       Let $\Delta s=\frac{\pi}{2}-s$.
                       Note that, for $n\geq 3$, 
                       we have $\frac{h}{n-1}-\frac{h}{n-1}\sin^{n-1}s=\frac{h}{2} \cdot (\Delta s)^2+$ high order terms
                       and $\cos s=\Delta s+$ high order terms.
                       So \eqref{n4} with plus sign becomes
                       $$
                       \lim_{s\uparrow \frac{\pi}{2}} 
                                       \dot s_1=
                                        \lim_{\Delta s\downarrow 0} 
                                                      \sqrt{
                                                               \dfrac{ \left( \frac{h}{2}\right)^2 (\Delta s)^4}
                                                               {(\Delta s)^2\big[(\Delta s)^2- \left( \frac{h}{2}\right)^2(\Delta s)^4\big]}
                                                         }
                                                         =
                                                         \frac{h}{2}.
                       $$
                       \end{proof}
                       
                       Lemma \ref{l1} leads to two useful consequences.
                       The first is to ensure that the $s_1$-value along the peculiar $(-h)$-curve for $C=\frac{h}{n-1}$  increases  to a finite number as $s\uparrow \frac{\pi}{2}$.
                       The second is to help understanding the tangential behavior in the limit.
                       
                         \begin{lem}\label{l2}
                         When approaching $p=(0,0,1)$, every $(-h)$-curve in local model is tangential to a geodesic $(\cos s\cdot q, \sin s)$ for some suitable $q\in \mathbb S^1\subset \mathbb C$.
                              \end{lem}
                              \begin{proof}
                              Suppose that the $(-h)$-curve in local model ends up with a limiting phase $\theta$.
                              Set  $q=e^{i\theta}$.
                              Denote $(\cos s\cdot q, \sin s)$ for $s\in [0, \frac{\pi}{2}]$ by $\Gamma_q$.
                              Then
                              in $\mathbb S^2$ it can be observed that, when $\Delta s\downarrow 0$,
                               the quantity 
                               $(\cos s)\cdot \dot s_1=(\sin\Delta s)\cdot \dot s_1$
                             approaches the  slope of the $(-h)$-curve with respect to $\Gamma_q$.
                              By the finiteness of $\dot s_1$ in Lemma \ref{l1} one can easily read it out that the limit of the slope  
                              is zero at $p$.
                              \end{proof}
                              
                              As a result, up to a phase gauge, there is a unique $(-h)$-curve in local model  limiting to the point $p$ tangential to some geodesic through $p$ for every $h>0$.
                              In fact, the local $(-h)$-curve can develop to a global one which is extremely natural in view of generating curves in $\mathbb S^2$.
                              However, the extension or, perhaps more precisely, assembling of local pieces at $p$ is a bit different but still very geometric.
                              \begin{thm}\label{flower}
                              For every $h>0$, there exists a unique global (complete) $(-h)$-curve passing through $p$.
                              Consequently, the Delaunay hypersurfaces of constant mean curvature $-h$ contain a new type
                              which serves as a bridge jointing the unduloid type  Delaunay hypersurfaces and noduloid type  Delaunay hypersurfaces.
                              \end{thm}
                              \begin{proof}
                              According to Lemma \ref{l2}, a normal vector field $N$ along $(-h)$-curve with $C=\frac{h}{n-1}$ in local model 
                              can have a $C^0$ limit $N_p$ at $p$.
                              Then we reflect the $(-h)$-curve at $p$ about the plane spanned by $\overrightarrow{0p}$ and $N_p$ in $\R^3$.
                             In this way  the local $(-h)$-curve $C^1$ extends through $p$ (with $C^0$ extension of $N$) as in the aerial view Figure \ref{fig:12} over point $p$.
                             Similarly, the extended $(-h)$-curve is $C^2$ at the joint point $p$.
                             Thus, the local $(-h)$-curve can extend repeatedly to a complete analytic $(-h)$-curve
                             which generates an $(-h)$-CMC immersion from $\R\times  S^{n-2}$ into $\mathbb S^{n}$.
                             The immersed hypersurface forms a flower type Delaunay hypersurface.
%
     %
           %
                                                           \end{proof}
                               \begin{figure}[ht]  
                                	\includegraphics[scale=0.75]{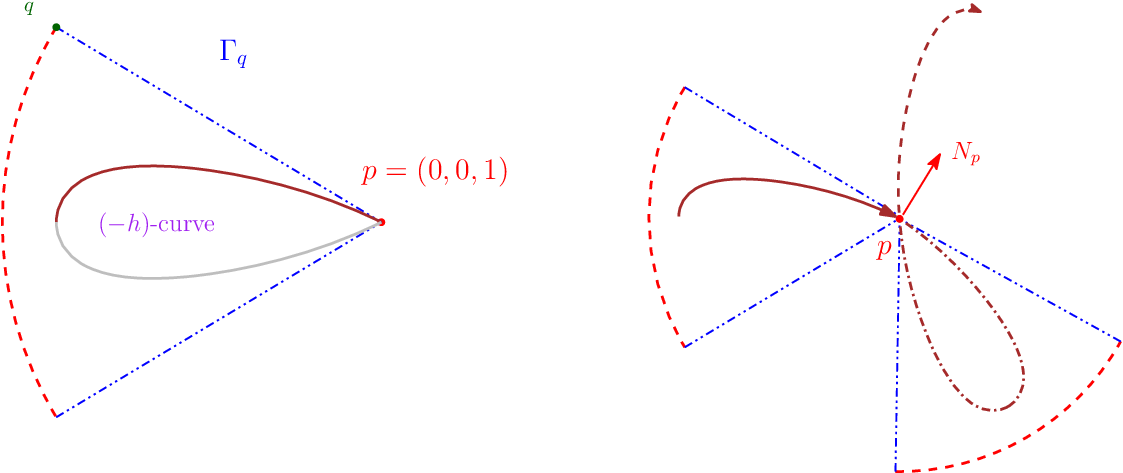}
                             \caption{Missing (flower) type which crosses $p$}
                               \label{fig:12}
 \end{figure}
                           \begin{rem}
                           The  mean curvature of the flower type Delaunay hypersurface induced by the $(-h)$-curve extended through $p$ via reflection by $N_p$ can be verified crossing $\{0+0i\}\times \mathbb S^{n-2}$ as follows.
                           First, by Lemma \ref{l1},
                           $\tilde \eta_0$ in \eqref{te0} at $p$ (equal to $(e^{i\theta},0,\cdots , 0)$) corresponds to $N_p$.
                           Hence, all curvatures arising from the $\mathbb S^{n-2}$ factor contributes nothing to the mean curvature of the flower type at $p$.
                           
                           So one only needs to compute the curvature $\kappa$ of the $(-h)$-curve at $p$.
                           Since approaching $p$ limits to a curve in polar coordinate of $\R^2$, namely $\rho:=\Delta s=\Delta s(s_1)=\rho(s_1)$,
                           it follows that, with respect to $-N$,
                           the curvature is given by
                            $$
                            \kappa(s_1)=\dfrac{\rho^2+2\left(\frac{d\rho}{ds_1}\right)^2-\rho\,\frac{d^2\rho}{ds_1^2}}
                            {
                                  \left\{
                                          \rho^2+
                                                        \left(
                                                                   \frac{d\rho}{ds_1}
                                                                    \right)^2
                                                                            \right\}
                                                                            ^{\frac{3}{2}}
                                                                            }\, .
                            $$
                            As 
                            $\frac{d\rho}{ds_1}=\left(\frac{ds_1}{d\rho}\right)^{-1}$
                            and
                            $\frac{d^2\rho}{ds_1^2}=-{\frac{d^2s_1}{d\rho^2}}\left({\frac{d s_1}{d\rho}}\right)^{-3}$,
                            according to Lemma \ref{l1}
                            we can see that  the curvature 
                            %
                            of the  curve at $p=(0,\theta)$
                            is $2\big(\frac{h}{2}\big)^{-2}\big/\big(\frac{h}{2}\big)^{-3}=h$.
                           \end{rem}
                       
                       \begin{rem}\label{aerial}
                       For $h$ fixed, if we give up the insistence on the choice of unit normal vector field
                       and take an obvious ``continuous" choice $\tilde N$ associated with the deformation of $C$ instead.
                       By decreasing $C$, 
                       a vivid video of deformation of Delaunay hypersurfaces can be obtained.
                       The moving slides start from the static type by \eqref{n00} to the unduloid type;
                       then pass through the union of spheres to the nodoid type;
                       with the drip area enlarging, 
                       finally come to the flower type in Theorem \ref{flower};
                       after that, meet the  unduloid type
                        and end up with another static (corresponding to \eqref{n00} for $-h$) type.
                        All of these hypersurfaces have constant mean curvature $h$ with respect to $\tilde N$.
                         \begin{figure}[ht]  
                                	\includegraphics[scale=0.7]{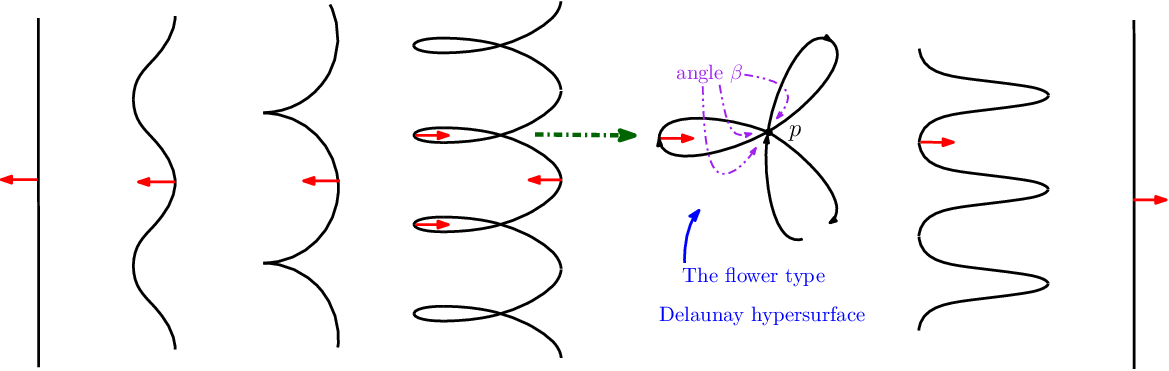}
                             \caption{Deformation of types by decreasing the value of $C$ for $h>0$}
                               \label{fig:13}
                                \end{figure}
                        However, note that in Figures \ref{fig:12} and \ref{fig:13}  illustrated petals are particularly chosen to be assembled by reflections along different $N_p$ repeatedly for an analytic $(-h)$-curve 
                        which induces an $(-h)$-CMC immersion of $\R\times  S^{n-2}$.
                        Actually,  the flower type can be viewed as a limit mapping from the point of view encoded in Definition \ref{tW} of adjusted width and Figure \ref{fig:14}.
                       \end{rem}
                       
                       As $C$ runs all possible values, Figure \ref{fig:13} exhausts all CMC Delaunay hypersurfaces when $h\neq 0$.
                       Hence, we finish completing the category of Delaunay hypersurfaces in $\mathbb S^n\subset \mathbb C\oplus \R^{n-1}$ 
                       through moving $\mathbb S^{n-2}\subset \R^{n-1}$ by a generating curve $\gamma \subset \mathbb S^3\subset \mathbb C\oplus \R$.
                       
                       \begin{cor}\label{complete}
                       For $h\neq 0$, up to a phase gauge in $s_1$ (or equivalently a rotation fixing $p=(0,0,1)$),
                       all Delaunay hypersurfaces of constant mean curvature $h$ with respect to some unit normal vector field in $\mathbb S^n$ are included in Figure \ref{fig:13}.
                       \end{cor}
                       
                       {\ }
                       
                       \section{Descendent to (immersed or embedded) closed CMC hypersurfaces}\label{S5}
                           In this section,  we  search for  flower type Delaunay hypersurfaces in $\mathbb S^n$ which induce closed immersed CMC hypersurfaces,
                          consider the descendent of \eqref{n1} to immersions of $S^1\times S^{n-2}$ into $\mathbb S^n$
                           and show the existence of uncountably many embedded Delaunay hypersurfaces of positive constant mean curvatures  in $\mathbb S^n$.

                           From \eqref{n4}, we have (valid for all types)
                            \begin{equation}\label{n40}
                                \dot s_1=
                                                   \dfrac{
                                                              C+\frac{h}{n-1}\sin^{n-1}s
                                                              }
                                                           {   \cos s \cdot \sqrt  {
                                                            %
                                                              \cos^2s \sin^{2(n-2)}s -\left(C+\frac{h}{n-1}\sin^{n-1}s\right)^2                                                   
                                                              }
                                                }\, \ .
              \end{equation}
Set $J(C, h)\subset (0,\frac{\pi}{2})$ to be the  largest interval for $\cos s \sin^{n-2}s -C+\frac{h}{n-1}\sin^{n-1}s>0$, cf. Figures \ref{fig:3} and \ref{fig:10}.
                           \begin{lem}\label{5l}
                           Except the situation of Proposition \ref{pp6}, i.e., $h\geq 0$ and $C=0$,
                           an $h$-curve 
                            forms a closed curve in $\mathbb S^3$ if and only if the width function of a period 
                          \begin{equation}\label{Wdef}
                          W(h,C):=2\int_{J(C, h)} \dot s_1 ds\in \pi \mathbb Q.
                          \end{equation}
                           \end{lem}
                           \begin{proof}
                           By the construction this is clear for curves not bothering $s=0$ or $s=\frac{\pi}{2}$.
                           For the flower type, let us set $\beta=W(-h, \frac{h}{n-1})$ as shown in Figure \ref{fig:13}.
                           Then the adjacent angle between two petals via one reflection is $\pi-2\beta$
                           and $2\beta$ by two repeated reflections.
                           The latter gives a $2\beta$-rotation about $p$.
                           Hence, the extended $(-h)$-curve of flower type is closed if and only if $\beta\in \pi \mathbb Q$.
                           \end{proof}
                           As illustrated in Figure \ref{fig:13}, let us mention the deformation of widths.
                           By considering line types in Figure \ref{fig:13} as limit situations,
                           $W(h,\cdot)$ is continuous in the first four and $W(-h,\cdot)$ in the last two pictures of Figure \ref{fig:13} respectively.
                           However, a gap occurs when $C=\frac{h}{n-1}$.
                           So, we introduce the concept of adjusted width according to geometric meaning of Figure \ref{fig:13}.
                                              \begin{figure}[ht]  
                                	\includegraphics[scale=0.86]{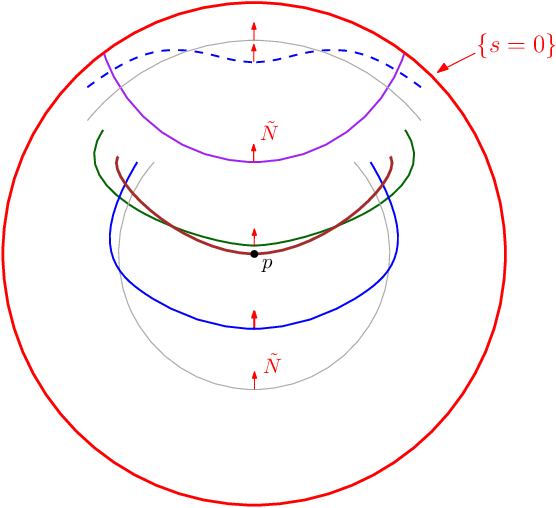}
                             \caption{Aerial view of  Figure \ref{fig:13} for fixed $h>0$ with respect to $\tilde N$}
                               \label{fig:14}
                                \end{figure}
                           \begin{defn}\label{tW}
                           Given $h\geq 0$, the adjusted width $\tilde W(h, \cdot)$ over $(-C_{-h}, C_h)$ is defined as
                           \begin{equation}\label{tw}
\tilde W(h, C):=\left\{
\begin{array}{cc}
W(h, C) & \ \ \ \ \ \ \ \ \ \ C\in ( -\frac{h}{n-1}, C_h),\\
\pi- W(-h, \frac{h}{n-1}) &   \ \ \ \ \ \ \ \ \ \  C=-\frac{h}{n-1},\\
2\pi-W(-h,-C) &   \ \ \ \ \ \ \ \ \ \  C\in ( -C_{-h}, -\frac{h}{n-1}).
\end{array}\right.
                           \end{equation}
                           \end{defn}
                         %
                           %
                                               \begin{lem}\label{twcinh}
                                                $\tilde W(h, \cdot)$ is continuous.
                                                \end{lem}
                                                \begin{proof}
                                                As the deformation moves,
                                                there is an amount $\pi$ jumping in both sides of the flower type.
                                                This is because the derivative of $\frac{h}{n-1}\sin^{n-1}s$ vanishes at $s=\frac{\pi}{2}$
                                                and the computation in Lemma 10.1 of \cite{L-Z} is still under control and can be adjusted to measure the jumps to be $2\times \frac{\pi}{2}$.
                                                \end{proof}
\begin{rem}\label{analytic}
In fact, the point $p=(0,0,1)$ is a singular point with respect to the parametrization $(s, s_1)$ not the geometric problem itself.
There is no problem at all for generating curves passing through $p$ and the continuity can be strengthened to be analytic, cf. Figure \ref{fig:14}, whenever the genuine singular set $\{s=0\}$ is avoided.
\end{rem}
  \begin{rem}
                           The adjusted width function actually exhibits a way to construct the flower type as a limit from its either side, see Figure \ref{fig:14}.
                           This perspective is different from the assembling thought in the proof of Theorem \ref{flower}.
                           \end{rem}
                             \begin{cor}
                                                $\tilde W(\cdot, \cdot)$ is continuous.
                                                \end{cor}
                                                \begin{proof}
                                                Since the splitting of domain and expressions over each piece are continuous in $h$, the statement follows by virtue of Lemma \ref{twcinh}.
                                                \end{proof}
                           
                           \begin{thm}\label{closedflower}
                           There are infinitely many $-h$ with $h>0$ for each of which the flower type Delaunay hypersurface can induce a $(-h)$-CMC immersion from $S^1\times S^{n-2}$ into $\mathbb S^n$.
                           \end{thm}
                           \begin{proof}
                           Note that the expression \eqref{Wdef} contributes nothing for $C=h=0$.
                           So $\tilde W(0,0)=\pi$.
                          But, for $C=\frac{h}{n-1}>0$,  it can be seen that $W(-h, \frac{h}{n-1})>0$
                          and consequently $\tilde W(h, -\frac{h}{n-1})\neq \pi$.
                          Hence,
                          by Lemma \ref{5l} and Definition \ref{tw},
                           every element in $\pi \mathbb Q\bigcap \big(\tilde W(h, -\frac{h}{n-1}),\, \pi\big)$ can decide an immersion of $S^1\times S^{n-2}$ into $\mathbb S^n$ of corresponding constant mean curvature.
                           \end{proof}
                                     
                                     \begin{thm}\label{closedunduloid}
                                     There exists a constant $c>0$ such that for every $h\in (0, c)$
                                     there are countably many 
                                     $(-h)$-CMC immersions from $S^1\times S^{n-2}$ into $\mathbb S^n$ induced by  the unduloid type Delaunay hypersurfaces (in the second last picture of Figure \ref{fig:13}).
                                     \end{thm}      
                                     \begin{proof}    
                                     Similar to the proof of Theorem \ref{closedflower}.
                                     When $C=h=0$, we have $\tilde W(0,0)=\pi$.
                                     For $C=-C_{0}$, by Lemma 10.1 of \cite{L-Z} or the argument (valid for $h=0$) of Lemma \ref{ppwidth} in below
                                      it can be computed that $\tilde W(0,-C_0)=(2-\sqrt 2)\pi$.
                                      Hence, by continuity there exists some open interval $(0, c)$ 
                                      such that $\tilde W(h,0)> \tilde W(h,-C_h)$ for every $h\in(0,c)$.
                                      So, for each such $h$, there exists some $C\in(-C_h, 0)$ for $\tilde W(h,C)\in \pi\mathbb Q$.
                                      By Lemma \ref{5l} and Definition \ref{tw},
                                       $\tilde W(h,C)\in \pi\mathbb Q$ implies
                                       the corresponding Delaunay hypersurfaces decides $(-h)$-CMC immersions from $S^1\times S^{n-2}$ into $\mathbb S^n$.
                                     \end{proof}  
                                     \begin{cor}
                                     The same statement of Theorem \ref{closedunduloid} holds for   the noduloid type Delaunay hypersurfaces.
                                     \end{cor}   
                                     \begin{proof}
                                    This follows by the argument in the proof of Theorem \ref{closedunduloid} and the analyticity  explained in Remark \ref{analytic}.
                                     \end{proof}                      
          
          Apparently, one cannot  expect embedded CMC hypersurface of flower type or noduloid type.
          Next, we focus on the second picture in Figure \ref{fig:13} and search for embedded CMC Delaunay hypersurfaces.

                                                                                                    \begin{lem}\label{ppwidth0}
                                                           Given $h>0$, $\lim_{C\downarrow 0}W(h, C)
                                                           =2\arctan\frac{n-1}{h}$.
     \end{lem}
                                \begin{proof}
                                According to Proposition \ref{pp6}, 
                                $W(h, 0)=2 z^h_R$.
                                As the mean curvature the hypersphere in $\mathbb S^n$ is given by $(n-1)\cdot\frac{\cos z^h_R}{\sin z^h_R}$,
                                one has 
                                $z^h_R=\arctan\frac{n-1}{h}$.
                                                           \end{proof}

                                                           \begin{lem}\label{ppwidth}
                                                           Given $h>0$, $\lim_{C\uparrow C_h}W(h, C)
                                                           =
                                                           {2\pi}
                                                    {{
                                                    \big[
                                                   1+(n-2){\cot^2 s_h}
                                                   \big]^{-\frac{1}{2}}}
                                                   }$.
                                                           \end{lem}
                                                            \begin{proof}
                                                           Recall that
                                                            $L$ and $R$ share
                                                            the same slope 
                                                            at $s=s_h$.
                                                            So let us  consider their second derivatives
$$\ddot L=\left\{
                                   \Big[
                                             -\tan s+(n-2)\cot s
                                                   \Big]^2
                                                   -
                                                   \left[
                                                   \frac{1}{\cos^2 s}+\frac{n-2}{\sin^2 s}
                                                   \right]
                                                   \right\}
                                                   \cos s \sin^{n-2} s
                                                   $$
                                                   and 
                                                $$\ddot R=h
                                   \big[
                                             -\tan s+(n-2)\cot s
                                                   \big]
                                       \cos s \sin^{n-2} s.$$
                                        We can observe
                                        that, when $C\uparrow C_h$,
                                                     the part under square in \eqref{n40}
                                                     becomes
                                                     $$
                                                    \text{ small positive number } A_C
                                                    -
                                                      \left[
                                                   \frac{1}{\cos^2 s}+\frac{n-2}{\sin^2 s}
                                                   \right] 
                                                   \left(
                                                   \cos s \sin^{n-2} s
                                                   \right)^2
                                                   \cdot (\delta s)^2
                                                   +o(\delta s)^2
                                                     $$
                                                     where $\delta s$ is quite small and $A_C\rightarrow 0$ as $C\uparrow C_h$.
                                                     As 
                                                     $
                                                     C_h+\frac{h}{n-1}\sin^{n-1}s_h= \cos s_h \sin^{n-2} s_h$,
                                                     it follows that
                                         %
                                              \begin{eqnarray}
                                             && \lim_{C\uparrow C_h}W(h, C)
                                            =
                                                   2 \lim_{C\uparrow C_h} \int_{J(C,h)} \dot s_1 ds
                                                    \nonumber\\
                                                                                                        &=&
                                                                                                       2 \lim_{C\uparrow C_h} \mathlarger\int_{J(C,h)}
                                                                                                        \, \dfrac{\left(C+\frac{h}{n-1}\sin^{n-1}s\right)ds}
                                                                                                        {\cos s\sqrt{A_C - 
                                                    \big[
                                                   \frac{1}{\cos^2 s}+\frac{n-2}{\sin^2 s}
                                                   \big] \left(\cos s \sin^{n-2} s\right)^2
                                                   \cdot (s-s_h)^2
                                                   }}
 \nonumber\\
                             &=&
                                 \dfrac{2\pi\cos s_h\sin^{n-2}s_h}
                                                    {\cos s_h\sqrt{
                                                    \big[
                                                   1+(n-2){\cot^2 s_h}
                                                   \big]} \sin^{n-2} s_h
                                                   }
 %
 =  
 \dfrac{2\pi}
                                                    {\sqrt{
                                                    \big[
                                                   1+(n-2){\cot^2 s_h}
                                                   \big]}
                                                   }
\nonumber \,.
                                              \end{eqnarray}
Here we apply the evaluation of $\arcsin$ twice and the integral turns out to be independent on value of $A_C$.
                                                           \end{proof}
                                    
                                    \begin{thm}
                                    For any $n\geq 3$, 
                                    there are uncountably many positive number $h$
                                    for each of which there exists an embedded $h$-CMC Delaunay hypersurface in $\mathbb S^n$.
                                    \end{thm}
                                    \begin{proof}
                                    By Lemma \ref{ppwidth0}, 
                                    $\lim_{C\downarrow 0}W(h, C)<\pi$
                                    for $h>0$.
                                    Whereas, by Lemma \ref{ppwidth},
                                          $$\lim_{C\uparrow C_h}W(h, C)
                                          =
                                          \frac{2\pi}
                                          %
                                                    {\sqrt{
                                                   1+(n-2){\cot^2 s_h}
                                                   }
                                                   }=
                                                   \frac{2\pi}{\sqrt{2+\frac{h^2+h\sqrt{h^2+4(n-2)}}{2(n-2)}}}>\pi\,\,\text{ when }h<\sqrt{\frac{4(n-2)}{3}}.$$
                                                   So, for every $h\in (0,\sqrt{\frac{4(n-2)}{3}})$,
                                                   there exists some $C\in (0, C_h)$
                                                  solving  $W(h, C)=\pi$.
                                                   As a result, the generating curve closes up in two periods as an embedded curve,
                                                   and the induced $h$-CMC Delaunay hypersurface is embedded as well.
                                    \end{proof}
                                    \begin{rem}
                                    For any fixed $h>0$, when $n$ is sufficiently large, there can be guaranteed an embedded $h$-CMC Delaunay hypersurface in   each   $\mathbb S^n$.                               
                                    \end{rem}

                                          Instead of $\pi$, one can consider to fulfill different size of width for an embedded Delaunay hypersurface,
                                          for example  $W(h, C)=\frac{2\pi}{k}$ where $2<k\in \mathbb Z$.
                                          \begin{thm}
                                          For $3\leq n\leq 9$,
                                          one can have infinitely many $2< k\in \mathbb Z$
                                          for each of which there exists a nonempty open interval $Z(k):=\left(\frac{k(n-1)}{\pi}, (k^2-2)\sqrt{\frac{n-2}{k^2-1}}\right)$
                                          such that, for any $h\in Z(k)$, 
                                          the required relation $W(h, C)=\frac{2\pi}{k}$ to generate embedded Delaunay hypersurfaces (consisting of $k$ periods of oscillation)  is solvable.
                                          \end{thm}
                                          \begin{proof}
                                          To solve
                         \begin{equation}\label{ineqk1}
                                          \lim_{C\uparrow C_h}W(h, C)
                                          =
                                          \frac{2\pi}
                                          %
                                                    {\sqrt{
                                                   1+(n-2){\cot^2 s_h}
                                                   }
                                                   }=
                                                   \frac{2\pi}{\sqrt{2+\frac{h^2+h\sqrt{h^2+4(n-2)}}{2(n-2)}}}>\frac{2\pi}{k},
                           \end{equation}
                                              we get $h< (k^2-2)\sqrt{\frac{n-2}{k^2-1}}$.
                                              
                                              On the other hand, $\arctan x\leq x$ for $x\geq 0$.
                                              Therefore, in order for 
                               \begin{equation}\label{ineqk2}
                                            \lim_{C\downarrow 0}W(h, C)
                                                           =2\arctan\frac{n-1}{h}
                                                           <\frac{2\pi}{k},
                                        \end{equation}
                                                      it suffices to require 
                                                      $\frac{n-1}{h}<\frac{\pi}{k}$.
                                                      
                                                      For $Z(k)$ to be nonempty,
                                                      observe that the highest order of its ending points
                                                      are the same and the coefficients  
                                                      are $\frac{n-1}{\pi}$ and $\sqrt{n-2}$ in limit.
                                                      Hence, when $\frac{n-1}{\pi}<\sqrt{n-2}$, i.e., $3\leq n\leq 9$,
                                                      $Z(k)$ is never empty when $k$ is large enough.
                                                      
                                                      Furthermore, by continuity based on \eqref{ineqk1} and \eqref{ineqk2},
                                                      for each $h\in Z(k)$ there exists suitable $C$ depending on   the value of $h$
                                                      such that 
                                                      $W(h, C)=\frac{2\pi}{k}$ is satisfied.
                                                      Correspondingly, an embedded $h$-CMC Delaunay hypersurface is generated in $\mathbb S^n$.
                                          \end{proof}
         \begin{rem}
        By the above argument for $3\leq n\leq 9$, it follows that the union of $\{Z(k)\}_{k> 2}$ will contain an infinite tail $(d_0, \infty)$
with the property that,  for any $h>d_0$ there exist at least $1+\left[\frac{\pi(h-d_0)}{n-1}\right]$ many different embedded $h$-CMC Delaunay hypersurface(s) in $\mathbb S^n$.
        Here the symbol $[\cdot]$ means  taking the largest integer part of the input.
         \end{rem}
   
   To the author, it seems unlikely  that  the unduloid type in the second last picture of Figure \ref{fig:13}  can produce embedded $(-h)$-CMC Delaunay hypersurfaces in $\mathbb S^n$.
   
    {\ }
    
   \section*{Acknowledgement}
     This work is
                       partially supported by NSFC (Grant Nos. 12022109 and 11971352).
                       The author would like to thank Professor Frank Morgan for his interest,
                       and
                       ICTP for warm hospitality 
                       where some initial inspiration of the paper was generated  in December 2023.

     {\ }


\end{document}